\documentclass[11pt,english]{article}
\newcommand\m[1]{\vcenter{\hbox{$\scriptstyle #1$}}}

\usepackage{amssymb,amsthm, amsmath, amsfonts, url, graphicx,verbatim,scrextend}
\usepackage{amscd}
\usepackage{enumitem}

\theoremstyle{plain}
\newtheorem*{thm*}{Theorem}
\theoremstyle{plain}
\newtheorem{thm}{Theorem}[section]
\theoremstyle{definition}
\newtheorem{defn}[thm]{Definition}
\theoremstyle{plain}
\newtheorem{lem}[thm]{Lemma}
\theoremstyle{plain}
\newtheorem{prop}[thm]{Proposition}
\theoremstyle{plain}
\newtheorem{cor}[thm]{Corollary}
\theoremstyle{remark}

\theoremstyle{remark}
\newtheorem*{acknowledgement*}{Acknowledgement}

\title{Isometric immersions of Riemannian manifolds in $k$-codimensional Euclidean space}
\author{Dan Gregorian Fodor}
\date{}

\begin{document}
\maketitle

\pagenumbering{roman}

\pagestyle{myheadings}\markboth{}{}

\pagenumbering{arabic}
\pagestyle{myheadings}
\begin{abstract}
We use a new method to give conditions for the existence of a local isometric immersion of a Riemannian $n$-manifold $M$ in $\mathbb{R}^{n+k}$, for a given $n$ and $k$. These equate to the (local) existence of a $k$-tuple of scalar fields on the manifold, satisfying a certain non-linear equation involving the Riemannian curvature tensor of $M$. Setting $k=1$, we proceed to recover the fundamental theorem of hypersurfaces.
In the case of manifolds of positive sectional curvature and $n\geq 3$, we reduce the solvability of the Gauss and Codazzi equations to the cancelation of a set of obstructions involving the logarithm of the Riemann curvature operator. The resulting theorem has a structural similarity to the Weyl-Schouten theorem, suggesting a parallelism between conformally flat $n$-manifolds and those that admit an isometric immersion in $\mathbb{R}^{n+1}$.
\end{abstract}

\begin{description}[leftmargin=8em]
\item [2010 Mathematics Subject Classification] 53A07.

\item[Keywords and phrases] Isometric immersions; Curvature equation; Weyl-Schouten-type theorem.
\end{description}

\section*{Introduction}
The problem of immersions of manifolds into Euclidean space is well studied. The Stiefel-Whitney classes \cite{Hatcher}, Whitney and Nash embedding theorems \cite{HH}, and the Cartan-Janet theorem \cite{HH} give us lower bounds on the number of dimensions required to embed various classes of manifolds into Euclidean space.  Here we look into the inverse problem: if we fix the dimension $m$ of the Euclidean ambient space in which a (local) isometric embedding exists, what can that tell us about the metric of our manifold? This acts as a sort of measure for the complexity of the metric, with the simplest metric the lowest number of dimensions: an $n$-manifold admits an isometric immersion in $\mathbb{R}^{n}$ if and only if its metric is flat.

In section $2$ we rewrite the (local) immersion of an $n$-manifold $(M,g)$ in $\mathbb{R}^{n+k}$ as a section of the $\mathbb{R}^{k}$ bundle over a flat $n$-manifold $(M,f)$. Thus, the existence of a local immersion becomes equivalent to the existence of a $k$-tuple of scalar fields $h\m{\tau}$, ${\tau}\in \{1..k\}$ such that $f_{ab}= g_{ab}-h\m{\tau}_{;a}h\m{\tau}_{;b}$ is a positive-definite flat metric.

The flatness of $f$ is characterised by the vanishing of its curvature. We next apply formulas connecting the curvatures of $2$ metrics, from \cite[Theorem 4.1]{Fodor}, to explicitly write the curvature of $f$ in terms of $g's$ existing connection and curvature.
Thus we rewrite the existence of a local immersion in $\mathbb{R}^{n+k}$ as the existence of a $k$-tuple of scalars $h\m{\tau}$, satisfying a certain nonlinear equation involving $g's$ connection and curvature:
$$(h\m{\alpha}_{;pj}h\m{\beta}_{;ik}-h\m{\alpha}_{;pk}h\m{\beta}_{;ij})( \delta\m{\alpha\beta}-h\m{\alpha}_{;a}h\m{\beta}_{;b}g^{ab})^{-1}=R_{pijk},\label{eq q}$$
with $f$ positive definite (Theorem 2.3). This is derived by setting the formula for $f's$ curvature to $0$.

In section 3 we  study the previous equation for $k$=1, obtaining:
$$\frac{h_{;pj}h_{;ik}-h_{;pk}h_{;ij}}{1-g^{ab}h_{;a}h_{;b}} = R_{pijk},$$
with positive definitness of $f$ becoming $g^{ab}h_{;a}h_{;b} < 1$. Thus we characterize the local existence of an immersion of $(M,g)$ in $\mathbb{R}^{n+1}$ as the existence of a scalar field $h$ satisfying the above conditions. This $h$ is basically a \emph{height} function lifting the flat manifold $(M,f)$ into $\mathbb{R}^{n+1}$. We show the equivalence of the above criteria to the Gauss and Codazzi equations, setting  $\Pi_{ab}= {h_{;ab} }/{(1-g^{ab}h_{;a}h_{;b})^{1/2}}$. The Codazzi equation $\Pi_{a[b;c]} = 0$ becomes an integrability condition allowing us to recover $h$ form  $\Pi_{ab}$ satisfying the Gauss equation and a set of initial conditions. We  then recover the fundamental theorem of hypersurfaces (Theorem 3.2).

In section 4 we rewrite the Gauss and Codazzi equations in a new form for manifolds of positive curvature. Considering ${\mathbb{R}_{ab}}^{cd}$ as an operator on the space of $2$-forms, we show there exists $\Pi_{ab}$ satisfying $$\Pi_{ac}\Pi_{bd}-\Pi_{ad}\Pi_{bc}=R_{abcd}$$ if and only if the Weil component of $R^{*}_{abcd}$ vanishes, with ${\mathbb{R}^{*}_{ab}}^{cd} = ln({\mathbb{R}_{ab}}^{cd})$ (the logarithm of the curvature operator). Due to positive curvature, $\Pi_{ab}$ turns out to be uniquely defined as
$\Pi_{ab}= \pm e^{P^{*}_{ab}}$, where $P^{*}_{ab}$ is the Schouten component of  ${\mathbb{R}^{*}_{ab}}^{cd}$ (Theorem 4.4).
We strengthen these results and rewrite the fundamental theorem of hypersurfaces (which characterises immersability in an Euclidian subspace of codimension $1$) in a form analogous to the Weyl-Schouten theorem (which characterises conformal  flatness). This is done in section 5 as Theorem 5.1.

In section 6 we study submanifolds of $(M,g)$ immersed in $\mathbb{R}^{n+1}$ obtained by intersecting the immersion of $M$ with an $n$-plane. These are just the areas of constant $h$ for particular $h$. We obtain a formula for the Riemannian curvature of these manifolds and find it equal to the curvature of the ambient manifold multiplied by a scaling factor greater than or equal $1$. Thus, if we study the cross-sections of the immersion (in $\mathbb{R}^{n+1}$) of an $n$-sphere with hyperplanes, they will be $(n-1)$-spheres with greater (or equal) positive curvature. If we immerse manifolds of  negative curvature, the cross-sections will be manifolds with greater or equal negative curvature. The cross-sections of immersions of flat manifolds will be flat manifolds.

\section{Notational conventions}
Since we use $m$-tuples of scalars, (or sections of an $\mathbb{R}^{m}$ bundle over the $n$-manifold), we need a way to index their elements. In addition to lower indices $V_a$, signifying covariant tensors, and upper indices $V^a$, signifying contravariant tensors, we shall also use middle indices $V\m{\alpha}$ to signify $m$-tuples, or sections of the $\mathbb{R}^{m}$ bundle. That is, $V\m{\alpha} = (V\m{1},V\m{2},...V\m{m})$. This gives us two additional rules: just as lower indices contract only with upper indices and vice-versa, so do middle indices contract only with middle indices, contraction signifying the ordinary cartesian product over $m$-vectors in $\mathbb{R}^{m}$. Eg. $V\m{\eta}K\m{\eta} = \sum_{\eta=1}^{m}V\m{\eta}K\m{\eta}$. Additionally, if a tensor has only middle indices, then its covariant derivative coincides with its ordinary derivative (this is a generalisation of the same rule for scalars, as such tensors can be seen as groupings of scalar-fields).

\section{Conditions for local immersions}
If we have an immersion $h:M\rightarrow \mathbb{R}^m$, of an $n$-manifold $M$ into $\mathbb{R}^{m}$, we can consider it to induce a metric on the manifold, namely the pull-back of the Euclidean metric along $h^{*}$. Associating to every point of the manifold the coordinate field $h\m{\tau}$, with $\tau \in\{1,2,...m\}$, the induced metric is $g_{ab}= h\m{\tau}_{,a}h\m{\tau}_{,b}$. Therefore, we have the following:

\begin{defn}
A Riemannian $n$-manifold $M$ admits a (local) isometric immersion in $\mathbb{R}^{m}$ if  there exists (locally) on the manifold an $m$-tuple of scalar fields, $h\m{\tau}$, satisfying $h\m{\tau}_{,a}h\m{\tau}_{,b} = g_{ab}$.
\end{defn}

The above formulation does not tell us very much about the relationships between the scalar fields and the intrinsic invariants of $M$. On the other hand, we know that an $n$-manifold admits a local immersion into $\mathbb{R}^{n}$ if and only if its curvature tensor vanishes. We would like a theorem that appears as a generalization of this particular case. The remainder of this section will be devoted to proving such a theorem.

\begin{lem} \label {1} 
A Riemannian $n$-manifold $(M,g)$ admits a local isometric immersion in  $\mathbb{R}^{n+k}$ if and only if there exists locally a $k$-tuple of scalars $h\m{\tau}$, such that $f_{ab}= g_{ab}-h\m{\tau}_{;a}h\m{\tau}_{;b}$ is a flat Riemannian metric.
\end{lem}
\begin{proof}
Assume there exists locally a $k$-tuple $h\m{\tau}$ such that  $f_{ab}=g_{ab}-h\m{\tau}_{;a}h\m{\tau}_{;b}$ is a flat metric. We can find a faithful (isometric) coordinate chart $m$ from $(M,f_{ab})$ to an open subset $M_f$ of $\mathbb{R}^{n}$. Consider the $\mathbb{R}^{k}$-bundle over $M_f$, and let $h\m{\tau}$ be a section of that bundle. The metric associated to $(m(p),h\m{\tau}(p))$, taken as a subset of $M_f\times \mathbb{R}^{k}$, is $f_{ab}+h\m{\tau}_{,a}h\m{\tau}_{,b} = f_{ab}+h\m{\tau}_{;a}h\m{\tau}_{;b} = g_{ab}$. However, since $M_f$ is a subset of $\mathbb{R}^n$, $M_f\times \mathbb{R}^{k}$ can be taken as a subset of $\mathbb{R}^{n}\times \mathbb{R}^{k}=\mathbb{R}^{n+k}$.  In conclusion, $(m,h\m{\tau})$ is a local, isometric immersion of $(M,g)$ in $\mathbb{R}^{n+k}$.

Conversely, assume a local isometric immersion $i:M_{i}\rightarrow \mathbb{R}^{n+k}$ exists, with $M_{i}\subset M$. Select a point $q\in M_{i}$ and choose normal coordinates at $q$. Let $p_{T}:M_{i}\rightarrow (T_{q}\cong \mathbb{R}^{n})$ be defined such that $p_T(r)$ is the orthogonal projection of  $i(r)$ onto the tangent bundle of $i(M_{i})$ at $i(q)$. Also, define $p_{N}:M_{i}\rightarrow (N_{q}\cong \mathbb{R}^{k})$, such that $p_N(r)$ is the orthogonal projection of $i(r)$ onto the normal bundle of $i(M_{i})$ at $i(q)$.

The identification of the tangent bundle with $\mathbb{R}^{n}$ and normal bundle with $\mathbb{R}^{k}$ can be made using a standard choice of normal coordinates. Since the differential of $p_{T}$ at $q$ is the linear isomorphism $\delta^{b}_a$, due to the inverse function theorem, there exists an open set $M_{q}\subset M_{i}$ on which $p_{T}$ is a bijection. The pullback metric induced by $p_{T}$ on $M_{q}$ is the metric $f_{ab}$ of $T_{q}$, which is flat.

By setting $h\m{\tau}(r) = p_{N}(r)$  for $r\in M_{q}$, the metric obeys the required equation: $g_{ab}=f_{ab}+h\m{\tau}_{;a}h\m{\tau}_{;b}$ for a flat $f_{ab}$, thus completing the proof.
\end{proof}

\subsection{Equation for local immersions}

\begin{thm}
A Riemannian $n$-manifold $(M,g)$ admits a local isometric immersion in  $\mathbb{R}^{n+k}$ if and only if there exists locally a $k$-tuple of scalars $h\m{\alpha}$, satisfying the following equation:
\begin{equation}(h\m{m}_{;pj}h\m{n}_{;ik}-h\m{m}_{;pk}h\m{n}_{;ij})(( \delta\m{mn}-h\m{m}_{;a}h\m{n}_{;b}g^{ab})^{-1})=R_{pijk},\label{eq q}\end{equation}
with $f_{ab}= g_{ab}-h\m{k}_{;a}h\m{k}_{;b}$ positive definite.
\end{thm}
\begin{proof}
Lemma 2.2 says the immersion exists if and only if there exists a $k$-tuple $h\m{\tau}$ such that $f_{ab}=g_{ab}-h\m{\tau}_{;a}h\m{\tau}_{;b}$ is a flat Euclidean metric.

Define $\Gamma^{a}_{bc} = \frac{1}{2} (f^{-1})^{an}(f_{nb;c}+f_{nc;b}-f_{bc;n})$ and\\  \hspace*{1.6cm}$R^{*l}_{ijk}=\Gamma^{l}_{ik;j} - \Gamma^{l}_{ij;k}+ \Gamma^{l}_{js}\Gamma^{s}_{ik} -  \Gamma^{l}_{ks}\Gamma^{s}_{ij}$.

According to \cite[Theorem 4.3]{Fodor}, the curvature of a secondary metric $f_{ab}$ on a manifold $M$ is  $R^{l}_{ijk}+R^{*l}_{ijk}$. As such,  $f_{ab}$ is flat if and only if $R^{l}_{ijk}+R^{*l}_{ijk} = 0$.

Substituting  $f_{ab}=g_{ab}-h\m{\tau}_{;a}h\m{\tau}_{;b}$, we obtain: $\Gamma^{a}_{bc} = -(f^{-1})^{an}h\m{\tau}_{;n}h\m{\tau}_{;bc}$ \\
Denote $T_{n}K_{n} = T_{a}K_{b}(f^{-1})^{ab}$ as a shorthand-contraction.\\ We have: \\
$\Gamma^{l}_{ik;j} = -(f^{-1})^{ln}_{;j}h\m{\tau}_{;n}h\m{\tau}_{;ik} - (f^{-1})^{ln}h\m{\tau}_{;nj}h\m{\tau}_{;ik}-(f^{-1})^{ln}h\m{\tau}_{;n}h\m{\tau}_{;ikj}$\\\\
$(f^{-1})^{ln}_{;j}=-(f^{-1})^{lq}f_{qm;j}(f^{-1})^{mn}=(f^{-1})^{lq}(h\m{\tau}_{;q}h\m{\tau}_{;m})_{;j}(f^{-1})^{mn}=\\\hspace*{1.2cm} =(f^{-1})^{lq}(h\m{\tau}_{;q}h\m{\tau}_{;mj}+h\m{\tau}_{;qj}h\m{\tau}_{;m})(f^{-1})^{mn}$.\\
This gives:\\
$\Gamma^{l}_{ik;j} = -(f^{-1})^{ln}(((h\m{\tau}_{;n}h\m{\tau}_{;jm})(h\m{p}_{;m}h\m{p}_{;ik})+(h\m{\tau}_{;nj}h\m{\tau}_{;m})(h\m{p}_{;m}h\m{p}_{;ik}))+\\\hspace*{3cm} +(h\m{\tau}_{;nj}h\m{\tau}_{;ik} +h\m{\tau}_{;n}h\m{\tau}_{;ikj}))$ and\\
$\Gamma^{l}_{jm}\Gamma^{m}_{ik}= (f^{-1})^{ln}(h\m{\tau}_{;n}h\m{\tau}_{;jm})(h\m{p}_{;m}h\m{p}_{;ik})$.\\\\
$\Gamma^{l}_{ik;j} + \Gamma^{l}_{jm}\Gamma^{m}_{ik}= -(f^{-1})^{ln}((h\m{\tau}_{;nj}h\m{\tau}_{;m})(h\m{p}_{;m}h\m{p}_{;ik})+(h\m{\tau}_{;nj}h\m{\tau}_{;ik} +h\m{\tau}_{;n}h\m{\tau}_{;ikj}))$ \\
$\Gamma^{l}_{ik;j} + \Gamma^{l}_{jm}\Gamma^{m}_{ik}= -(f^{-1})^{ln}((h\m{\tau}_{;nj}h\m{p}_{;ik})(h\m{\tau}_{;m}h\m{p}_{;m} + \delta\m{\tau p} )+h\m{\tau}_{;n}h\m{\tau}_{;ikj})$\\\\
$R^{l}_{ijk}+R^{*l}_{ijk} = 0  \Longrightarrow R^{l}_{ijk} + \Gamma^{l}_{ik;j} - \Gamma^{l}_{ij;k}+ \Gamma^{l}_{js}\Gamma^{s}_{ik} -  \Gamma^{l}_{ks}\Gamma^{s}_{ij} = 0 \Longrightarrow$\\
$R^{l}_{ijk} = (\Gamma^{l}_{ij;k} + \Gamma^{l}_{km}\Gamma^{m}_{ij}) - (\Gamma^{l}_{ik;j} + \Gamma^{l}_{jm}\Gamma^{m}_{ik}) $.\\\\
Multiplying by $f_{nl}$ and expanding the terms, we get:\\
$f_{nl}R^{l}_{ijk}=(h\m{\tau}_{;nj}h\m{p}_{;ik}- h\m{\tau}_{;nk}h\m{p}_{;ij})(h\m{\tau}_{;m}h\m{p}_{;m} + \delta\m{\tau p}) + h\m{\tau}_{;n}(h\m{\tau}_{;ikj} - h\m{\tau}_{;ijk})$\\\\
Applying the Ricci identity, we get:\\
$h\m{\tau}_{;n}(h\m{\tau}_{;ikj} - h\m{\tau}_{;ijk}) = h\m{\tau}_{;n}( h\m{\tau}_{;l}R^{l}_{ikj}) = -(h\m{\tau}_{;n} h\m{\tau}_{;l})R^{l}_{ijk}$\\
$(f_{nl}+h\m{\tau}_{;n} h\m{\tau}_{;l})R^{l}_{ijk}=(h\m{\tau}_{;nj}h\m{p}_{;ik}- h\m{\tau}_{;nk}h\m{p}_{;ij})(h\m{\tau}_{;m}h\m{p}_{;m} + \delta\m{\tau p})$\\
$g_{nl}R^{l}_{ijk}=(h\m{\tau}_{;nj}h\m{p}_{;ik}- h\m{\tau}_{;nk}h\m{p}_{;ij})(h\m{\tau}_{;m}h\m{p}_{;m} + \delta\m{\tau p})$\\
$R_{nijk}=(h\m{\tau}_{;nj}h\m{p}_{;ik}- h\m{\tau}_{;nk}h\m{p}_{;ij})(h\m{\tau}_{;m}h\m{p}_{;m} + \delta\m{\tau p})$
\begin{equation}
R_{nijk}=(h\m{\tau}_{;nj}h\m{p}_{;ik}- h\m{\tau}_{;nk}h\m{p}_{;ij})(h\m{\tau}_{;a}h\m{p}_{;b}(f^{-1})^{ab} + \delta\m{\tau p}). \label{eq:rnh}
\end{equation}
Now we prove $(h\m{\tau}_{;a}h\m{p}_{;b}(f^{-1})^{ab} + \delta\m{\tau p}) =( \delta\m{\tau p}-h\m{\tau}_{;a}h\m{p}_{;b}g^{ab})^{-1}$:\\\\
 $(h\m{\tau}_{;a}h\m{p}_{;b}(f^{-1})^{ab} + \delta\m{\tau p})( \delta\m{pq}-h\m{p}_{;a}h\m{q}_{;b}g^{ab})=$\\
 $=  \delta\m{\tau q} + h\m{\tau}_{;a}h\m{q}_{;b}(f^{-1})^{ab} - h\m{\tau}_{;a}h\m{q}_{;b}g^{ab} - h\m{\tau}_{;a}(f^{-1})^{ac}h\m{p}_{;c}h\m{p}_{;d}g^{db}h\m{q}_{;b} = $\\
 $=  \delta\m{\tau q} + h\m{\tau}_{;a}h\m{q}_{;b} ((f^{-1})^{ab} - g^{ab} - (f^{-1})^{ac}h\m{p}_{;c}h\m{p}_{;d}g^{db})=$\\
 $=  \delta\m{\tau q} + h\m{\tau}_{;a}h\m{q}_{;b} ((f^{-1})^{ac}(g_{cd})g^{db} - (f^{-1})^{ac}(f_{cd}) g^{db} - (f^{-1})^{ac}(h\m{p}_{;c}h\m{p}_{;d})g^{db})=$\\
$=\delta\m{\tau q} + h\m{\tau}_{;a}h\m{q}_{;b}(f^{-1})^{ac}g^{db}(g_{cd} - f_{cd} - h\m{p}_{;c}h\m{p}_{;d})= \delta\m{\tau q}$\\\\
Combining with \eqref{eq:rnh}, we get:
\begin{equation*}
R_{nijk}=(h\m{\tau}_{;nj}h\m{p}_{;ik}- h\m{\tau}_{;nk}h\m{p}_{;ij})( \delta\m{\tau p}-h\m{\tau}_{;a}h\m{p}_{;b}g^{ab})^{-1}, \end{equation*}
thus completing the proof.
\end{proof}

\noindent The algebraic relationships of \cite[Theorem 4.3]{Fodor} suffice to determine whether the curvature of a metric is $0$. However, in cases where the metric can be pseudo-Euclidean, they do not distinguish between Euclidean and pseudo-Euclidean (eg. Minkowski) flat metrics. We construct our immersion  as a section of the $\mathbb{R}^k$ bundle over a flat manifold $(M,f)$. The additional condition that $f$ be positive definite ensures that $M$ is Euclidean and not pseudo-Euclidean. Also, it is sufficient to check the condition at a single point. As $f$ is invertible, no changes of signature occur on its domain of definition. If we change the signature of $f$ in our theorem from $(k,0)$ to $(k-p,p)$, the existence of the $h\m{\alpha}, \alpha\in \{1,..., k\}$, scalar fields become instead conditions for the local immersion of our $n$-manifold in a pseudo-Euclidean space of signature $(n+k-p,p)$. Many metrics which do not admit local immersions in an Euclidean space $\mathbb{R}^{n+k}$ may instead admit immersions in a pseudo-Euclidean space of the same dimension.

\section{Immersion of $n$-manifolds in $\mathbb{R}^{n+1}$}
Due to its non-linear terms, obstructions to the solvability of \eqref{eq q} appear difficult to calculate. Here we will study the case when $k=1$, that is, we are immersing an $n$-manifold into $\mathbb{R}^{n+1}$. The $k$-tuples become singular scalars, and the equation simplifies to:
\begin{equation}\frac{h_{;pj}h_{;ik}-h_{;pk}h_{;ij}}{1-g^{ab}h_{;a}h_{;b}} = R_{pijk}.\label{eq p}\end{equation} The condition that $f_{ab}$ is positive-definite becomes $g^{ab}h_{;a}h_{;b} < 1$.
\begin{thm} \label{thm:R}
A Riemannian $n$-manifold $M$ admits a local isometric immersion in $\mathbb{R}^{n+1}$ if and only if there exists a local scalar field $h$, satisfying \eqref{eq p}, and  $g^{ab}h_{;a}h_{;b} < 1$.
\end{thm}
\newpage
\subsection{Recovering the fundamental theorem of hypersurfaces} \hspace{1cm}
\begin{lem}
Assume a Riemannian $n$-manifold $(M,g)$ has a local scalar field $h$ satisfying $R_{pijk} = \frac{h_{;pj}h_{;ik}-h_{;pk}h_{;ij}}{1-g^{ab}h_{;a}h_{;b}} $. Then $\Pi_{ab}= \frac{h_{;ab} }{(1-g^{ab}h_{;a}h_{;b})^{1/2}}$ satisfies $R_{pijk} = \Pi_{pj}\Pi_{ik}-\Pi_{pk}\Pi_{ij}$ (Gauss equation), and  $\Pi_{a[b;c]}=0$ (Codazzi equation).
\end{lem}
\begin{proof}
By substituting $\Pi_{ab}= \frac{h_{;ab} }{(1-g^{ab}h_{;a}h_{;b})^{1/2}}$ into the Gauss and Codazzi equations, we check they are satisfied.
\end{proof}

\begin{lem}
Assume a Riemannian manifold $M$ has on an open contractible subset $N$, a symmetric bilinear form $\Pi_{ab}$ satisfying $R_{pijk} = \Pi_{pj}\Pi_{ik}-\Pi_{pk}\Pi_{ij}$ and $\Pi_{a[b;c]}=0$. Then, for any point $p \in N$, and any initial conditions $h$, and $h_{;a}$ at $p$, with $h_{;a}h_{;b}g^{ab}< 1$, there is a maximal open connected set $S\subset N$, $p\in S$, on which there exists $h$ satisfying $\Pi_{ab}= \frac{h_{;ab} }{(1-g^{ab}h_{;a}h_{;b})^{1/2}}$, and the initial conditions at $p$. The field $h$ is unique. 
\end{lem}
\begin{proof}
The existence of a field of real symmetric bilinear forms $\Pi_{ab}$ satisfying $R_{pijk} = \Pi_{pj}\Pi_{ik}-\Pi_{pk}\Pi_{ij}$ gives the first obstruction to the existence of a height function $h$, satisfying  $\Pi_{ab}= \frac{h_{;ab} }{(1-g^{ab}h_{;a}h_{;b})^{1/2}}$. Its existence also guarantees that $g^{ab}h_{;a}h_{;b} < 1$  is automatically satisfied, as having $g^{ab}h_{;a}h_{;b} \geq 1$ would introduce imaginary factors in $\Pi$. The form $\Pi_{ab}$ will be shown to be  the second fundamental form of Gauss.

Given the existence of such a form, can we always locally construct $h$ satisfying $\frac{h_{;ab}}{(1-g^{ab}h_{;a}h_{;b})^{1/2}} = \Pi_{ab}$? The case of $3$-manifolds provides a counterexample, as $3$-manifolds with a given $\Pi_{ab}$ satisfying the Gauss equation can always be found, yet not all such manifolds admit immersion into $\mathbb{R}^4$. There exists an additional obstruction, which turns out to be the Codazzi equation.

We now shall study the obstructions to recovering $h$ from $\Pi$.
Consider a path $c^{n}:[0,1]\rightarrow M$ from a point $p$ to $q$, parameterized by natural parameter $t$, and fix $h_{;a}$ at $p$, as an initial condition with $g^{ab}h_{;a}h_{;b}<1$. Knowing   $\Pi_{ab} = \frac{h_{;ab}}{(1-g^{ab}h_{;a}h_{;b})^{1/2}}$, we can recover $h_{;ab}$ at $0$. Thus, we recover the $h_{;a}$ along the path by integrating the equation: 
\begin{equation}
\frac {D {h_{;a}}}{dt}  = (1-g^{mn}h_{;m}h_{;n})^{1/2} \Pi_{ab} \frac{d c^b}{dt}, \quad g^{ab}h_{;a}h_{;b}<1. \label{eq:dha}
\end{equation}
This gives us the maximal domain $S\subset N$ and the initial conditions for $h$, and $h_{;a}$ at $p$ guarantee uniqueness. 

A vector field $h_{a}$ satisfying  $\Pi_{ab} = \frac{h_{a;b}}{(1-g^{ab}h_{a}h_{b})^{1/2}}$ admits a local existence if and only if the result of such integration is independent of the path.
$h_{a;b}$ must obey the Ricci identity:  $h_{a;bc}-h_{a;cb} = h_{n}R^{n}_{abc}$. We do not know $h_{a;b}$, but only  $\Pi_{ab} = \frac{h_{a;b}}{(1-g^{ab}h_{a}h_{b})^{1/2}}$. As such, we will apply the anticommutator to the last two indices of $\Pi_{ab;c}$, and examine the result. \\ $\Pi_{a[b;c]} = \frac{h_{a;[bc]}}{(1-g^{ab}h_{a}h_{b})^{1/2}}+\frac{1}{2}((\frac{1}{(1-g^{ab}h_{a}h_{b})^{1/2}})_{;c}h_{a;b}- (\frac{1}{(1-g^{ab}h_{a}h_{b})^{1/2}})_{;b}h_{a;c})$.\\
Expanding $(\frac{1}{(1-g^{ab}h_{a}h_{b})^{1/2}})_{;c}$, we get $\frac{-2{g^{ab}h_{a}h_{b;c}}}{(1-g^{ab}h_{a}h_{b})^{3/2}}$. Substituting in $\Pi_{a[b;c]}$ and grouping the terms, we obtain: $\Pi_{a[b;c]} = \frac{h_{a;[bc]} - \frac{h_{n}g^{nm}(h_{m;b}h_{ac}-h_{m;c}h_{a;b})}{2(1-g^{ab}h_{a}h_{b})} }{(1-g^{ab}h_{a}h_{b})^{1/2}}$
$\Pi_{a[b;c]} = \frac{h_{a;[bc]} -\frac{1}{2} h_{n}{\mathbb{R}^n}_{abc}}{(1-g^{ab}h_{a}h_{b})^{1/2}}$.\\
The Ricci identity for $h_{a}$ holds if and only if $\Pi_{a[b;c]}=0$. Thus, $h_{a}$ exists if and only if $\Pi_{a[b;c]}=0$. Or, $\Pi_{ab}$ is a Codazzi tensor. This is the second obstruction to the local existence of $h$. However, since $\Pi_{ab}$ is symmetric, so is $h_{;ab}$. Therefore, if $h_{;a}$ exists, so does $h$.
\end{proof}
\subsubsection{Notes on the extensibility of solutions}  
Setting initial conditions $h$ and $h_{;a}$ at a given point, we can define and extend the $h$ field satisfying the above equations in any direction on a contractible set, until $h_{;a}h_{;b} g^{ab} = 1$. If we take $h_{;a}(p) = 0$ as initial condition on a point $p$, and set $|v^{a}\Pi_{ab}v^{b}| \leq r$ for $v^{a}g_{ab}v^{b} = 1$, then we can define the $h$-field on at least a ball of radius $\frac{\pi}{2r}$ around $p$. The equation $$\frac {D {h_{;a}}}{dt}  = (1-g^{mn}h_{;m}h_{;n})^{1/2} r \frac{d c^a}{dt},$$ obtained by maximising \eqref{eq:dha}, for a geodesic $c^{a}$, gives  $|h_{;a}| = \sin(rd)$ at a distance of $d$ from $p$.

Given a resulting immersion, this is also a lower bound for the injectivity radius of its projection to the hyperplane tangent to $p$. As we shall see in the following theorem, different initial conditions for $h$ at a point result in different immersions that are rigid transformations of one another, and are defined on different subdomains of $m$ . We can glue them together (applying transformations to their immersion coordinates to make them match)  to extend a solution past the injectivity radius. Note that the resulting $h$ of the extended solution might not always satisfy $\Pi_{ab}= \frac{h_{;ab} }{(1-g^{ab}h_{;a}h_{;b})^{1/2}}$ : there may be regions past the injectivity radius where it satisfies $\Pi_{ab}= - \frac{h_{;ab} }{(1-g^{ab}h_{;a}h_{;b})^{1/2}}$. Both these equations are equally valid solutions for $R_{pijk} = \frac{h_{;pj}h_{;ik}-h_{;pk}h_{;ij}}{1-g^{ab}h_{;a}h_{;b}}$ (see Theorem \ref{thm:R}). To understand this phenomenon, we shall see how $h$ behaves on the simple example of the circle immersed in $\mathbb{R}^{2}$. We have $f:[-\pi,\pi]\rightarrow \mathbb{R}^2$,  $f(x) = (\cos(x),-\sin(x))$, giving us $h(x) = -\sin(x)$. The circle has a constant principal curvature of $1$, giving us the equations: $h^{\prime\prime} = {\pm}(1-h^{\prime}h^{\prime})^{\frac{1}{2}}$ . We see that $h(x) = -\sin(x)$ satisfies the positive-signed equation on $[0,\frac{\pi}{2}]$ and the negative-signed equation on  $[-\frac{\pi}{2},0]$. At $0$ we have a region on which $|h^{\prime}|= 1$. By reversing the choice of sign of our second fundamental form, we either reverse the sign of $h$, or swap the domains on which these two variants of the equation are satisfied.\\

\noindent Next we present a new proof for the fundamental theorem of hypersurfaces, \cite[Theorem 7.1]{KN}.

\begin{thm}  {\bf Fundamental theorem of hypersurfaces:}\\
A Riemannian $n$-manifold $M$ admits a local isometric immersion in $\mathbb{R}^{n+1}$ if and only if there exists locally a symmetric $\Pi_{ab}$ satisfying ${R_{abcd}}$ = $\Pi_{ac}\Pi_{bd}-\Pi_{ad}\Pi_{bc}$ and $\Pi_{a[b;c]}=0$. For any given $\Pi_{ab}$ satisfying these equations, there exists a local immersion, unique up to rigid transformations, such that $\Pi_{ab}$ acts as the second fundamental form of said immersion.
\end{thm}
\begin{proof}
By choosing a point $p\in M$ and setting $h_{;a}(p)$ and $h(p)$ as initial conditions, the given $\Pi_{ac}$ satisfying the Gauss and Codazzi equations allows us to construct a unique $h$-field in a neighborhood $M_{p}$ of $p$, satisfying  $ \frac{h_{;ab} }{(1-g^{ab}h_{;a}h_{;b})^{1/2}}=\Pi_{ab}$, as shown if Theorem 3.1. The resulting immersion will have the form $i:M_{p}\rightarrow \mathbb{R}^{n+1}$, $i(x)=(f\m{\tau}(x),h(x))$, with $\tau\in\{1,2,..n\}$, where $f\m{\tau}:M_{p}\rightarrow \mathbb{R}^{n}$ is a local isometric immersion of $(M_{p}, f_{ab})$ into $\mathbb{R}^{n}$,  with $f_{ab}$ being the flat metric $f_{ab}=g_{ab}-h_{;a}h_{;b}= f\m{\tau}_{;a}f\m{\tau}_{;b}$, and associated unit-normal field $(1-g^{ab}h_{;a}h_{;b})^{\frac{1}{2}}(-f\,{\tau}_{;p}(f^{-1})^{pk}h_{;k},1)$. The proof consists of two steps:
first we show that for an immersion of the form  $i(x)=(f\m{\tau}(x),h(x))$, the second fundamental form satisfies $\Pi_{ab} = \frac{h_{;ab} }{(1-g^{ab}h_{;a}h_{;b})^{1/2}}$, then we show that "fixing" an immersion of $M_{p}$ with given $\Pi_{ab}$, in $\mathbb{R}^{n+1}$ using a rigid transformations that identifies the point $p$ with a point $i(p)\in \mathbb{R}^{n+1}$ and a normal and $n$-frame at $p$ with an $(n+1)$-frame at $i(p)$ is equivalent to picking the initial conditions of $h(p)$, $h_{;k}(p)$ and $f\m{\tau}:(M_{p}, f_{ab})\rightarrow \mathbb{R}^{n}$ that uniquely determine the solution to the equation. As such, all local isometric immersions of $(M_{p},g)$ in $\mathbb{R}^{n+1}$ with a the same given (scalar-valued) fundamental  form $\Pi_{ab}$ are rigid-transformations of one-another.\\
\noindent\rule{8cm}{0.4pt}
\begin{lem}
For a given isometric immersion $I\m{\upsilon}:(M,g)\rightarrow \mathbb{R}^{n+k}$ with $\upsilon  \in\{1,2..n+k\}$, the second fundamental form (taken as a vector-valued form with values in the normal-bundle as  sub-bundle of $\mathbb{R}^{n+k}$) takes the form $\Pi\m{\upsilon}_{ab}=I\m{\upsilon}_{;ab}$.
\end{lem}
\begin{proof}
Let $V^{a}$ be a vector-field and $V\m{\upsilon} = V^{a} I\m{\upsilon}_{;a}$ be its pushforward in the ambient-bundle of the immersion. We have $V\m{\upsilon}_{;b} =  V^{a}_{;b} I\m{\upsilon}_{;a} + V^{a} I\m{\upsilon}_{;ab}$. As  $I\m{\upsilon}_{;a}$  acts as a pushforward from the tangent bundle of the manifold to the tangent bundle of the immersion, using the Gauss formula, we identify  $V_{a;b} I\m{\upsilon}_{;a}$ as the push-forward of the Levi-Civita connection applied to $V^{a}$, and $I\m{\upsilon}_{;ab}$ as the second fundamental form, with values in the normal bundle.
\end{proof}

\noindent\rule{8cm}{0.4pt}\\
\noindent Applying to our $i\m{\epsilon}:M_{p}\rightarrow \mathbb{R}^{n+1}$, $i(x)=(f\m{\tau}(x),h(x)),\epsilon\in\{1,2,....n+1\}$ we get $\Pi\m{\epsilon}_{ab}=(f\m{\tau}_{;ab},h_{;ab})$, $\tau\in\{1,2,....n\}$ as vector-valued second fundamental form.

It can be checked that $V\m{\epsilon} = (1-g^{ab}h_{;a}h_{;b})^{\frac{1}{2}}(-f\,{\tau}_{;p}(f^{-1})^{pk}h_{;k},1)$ is a valid choice of unit-normal field , as $V\m{\epsilon}V\m{\epsilon} = 1$ (unit) and $({V\m{\epsilon}})(i\m{\epsilon}_{;p}) = 0$ (normal). We recover our scalar-valued second-fundamental form as $\Pi_{ab} =  \Pi\m{\epsilon}_{ab}V\m{\epsilon}$. This gives  $\Pi_{ab}= (h_{;ab} - (f\m{\tau}_{;ab}f\,{\tau}_{;p})(f^{-1})^{pk}h_{;k})(1-g^{ab}h_{;a}h_{;b})^{\frac{1}{2}}$. (3)\\
We have $f\m{\tau}_{;ab}f\,{\tau}_{;p}=\frac{1}{2}((f\m{\tau}_{;a}f\m{\tau}_{;p})_{;b}+(f\m{\tau}_{;b}f\m{\tau}_{;p})_{;a}-(f\m{\tau}_{;a}f\m{\tau}_{;b})_{;p} )= \\ \hspace*{3cm}=-\frac{1}{2}((g_{ap}-h_{;a}h_{;p})_{;b}+(g_{bp}-h_{;b}h_{;p})_{;a}-(g_{ab}-h_{;a}h_{;b})_{;p})=\\ \hspace*{3cm}=\frac{1}{2}((h_{;a}h_{;p})_{;b}+(h_{;b}h_{;p})_{;a}-(h_{;a}h_{;b})_{;p} )= (-h_{;ab}h_{;p})$ \hspace{0.25cm} (4)\\
Substituting (4) in (3) gives:\\ $\Pi_{ab}= (h_{;ab} + h_{;ab}h_{;p}(f^{-1})^{pk}h_{;k})(1-g^{ab}h_{;a}h_{;b})^{\frac{1}{2}} =  \\\hspace*{0,6cm}=h_{;ab}(1+ h_{;p}h_{;k}(f^{-1})^{pk})(1-g^{ab}h_{;a}h_{;b})^{\frac{1}{2}} $ \hspace*{0.25cm} (5) \\Equation (2) from the proof of theorem 2.3 gives:\\ $(1+ h_{;p}h_{;k}(f^{-1})^{pk}) = (1-g^{ab}h_{;a}h_{;b})^{-1}$. \\ Substituting in (5) gives:\\
$\Pi_{ab}=\frac{h_{;ab} }{(1-g^{ab}h_{;a}h_{;b})^{1/2}}$, as expected. This completes the first part of the proof.

Consider an immersion $i:M_{p}\rightarrow \mathbb{R}^{n+1}$, split as $i(x)=(f\m{\tau}(x),h(x))$. Let $(v^1,v^2,v^3...v^n)$ be an orthonormal frame at $p\in M_{p}$, and require that $i(p)=(0,0...0)$, $ \nabla i(v^k)=e_{k}$, the standard basis at $p$, and $e_{n+1}$ be the unit-normal vector of the immersion at $i(p)$. Knowing that $e_{n+1}= (0\m{\tau},1)$ is the unit-normal at $i(p)$ gives $h_{;k}(p)=0$. Additionally, $i(p)=(0,0...0)$ gives $h(p)=0$. As $\Pi_{ab}$ is given, this determines a unique $h$-field. We have $g_{ab}(p)=f_{ab}(p)$ and $i\m{\epsilon}_{;k}(p) =(f\m{\tau}_{;k}(p),h_{;k}(p))= (f\m{\tau}_{;k}(p),0_{;k})$. The immersion $f\m{\tau}:(M_{p}, f_{ab})\rightarrow \mathbb{R}^{n}$ is required to take $p$ to $(0,0...0)$ and the frame $(v^1,v^2,v^3...v^n) \in T_{p}M$ to $(e_1,e_2,e_3...e_n)\in T_{0}\mathbb{R}^n$. As such, it is uniquely determined, and so is $i\m{\epsilon}$ for the given $\Pi_{ab}$ satisfying the Gauss and Codazzi equations. This completes the proof.
\end{proof}

\section{Reducing the equation to obstruction-form}
In this section we derive necessary and sufficient conditions for the existence of $\Pi_{ab}$ satisfying the Gauss and Codazzi equations (hence the existence of local immersions) as the vanishing of certain obstructions, in manifolds of positive-sectional curvature and $n\geq 3$.

\begin{prop}
For a given tensor $R_{pijk}$, possessing the symmetries of the curvature tensor, there exists an $n$-manifold  $M$ immersed in $\mathbb{R}^{n+1}$, such that $R_{pijk}$ is equal to the Riemannian curvature field of $M$ at a given point, if and only if there exists a symmetric tensor $\Pi_{ab}$ such that  our tensor satisfies: $$\Pi_{pj}\Pi_{ik}-\Pi_{pk}\Pi_{ij}=R_{pijk}.$$
\end{prop}
\begin{proof}
The fact that an immersion requires the existence of a $\Pi_{ab}$ tensor was proven above. For the converse, use the $\Pi_{ab}$ tensor to construct the immersion given by: $f:\mathbb{R}^{n}\rightarrow \mathbb{R}^{n+1}$, $f(v^n) = (v^n, \frac{\Pi_{ab}v^av^b}{2})$. The induced metric will have the required curvature tensor $R_{pijk}=\Pi_{pj}\Pi_{ik}-\Pi_{pk}\Pi_{ij}$ at coordinates $v^n=0$.
\end{proof}

\begin{lem}
A sectionally-positive tensor $R_{pijk}$ admitting a form  $R_{pijk}=\Pi_{pj}\Pi_{ik}-\Pi_{pk}\Pi_{ij}$ is positive-definite (as an operator on the space of $2$-forms).
\end{lem}
\begin{proof}
Let $(v^1,v^2,....v^n)$ be an orthonormal basis that diagonalizes $\Pi_{ab}$, with eigenvalues $(\lambda_1, \lambda_2,....\lambda_n)$. Then the eigenvectors of $R_{pijk}$ can be written as $\{w^{kp}= v^k \wedge v^p \mid  k,p\in \{1,2...n \}, k<p \}$, with eigenvalues $\lambda_k\lambda_p$ for $w^{kp}$. All the eigenvectors are decomposable $2$-forms, therefore positive sectional curvature implies positive-definite curvature.
\end{proof}

\begin{cor}
A sectionally positive curvature tensor that is not positive-definite requires an Euclidean space of dimension at least $n+2$ for the existence of an isometric immersion.
\end{cor}
\begin{proof}
Assume dimension  $n+1$ is sufficient for an isometric immersion. Then, according to Proposition 4.1, there exists a symmetric tensor $\Pi_{ab}$ satisfying $\Pi_{pj}\Pi_{ik}-\Pi_{pk}\Pi_{ij}=R_{pijk}$. According to Lemma 4.2, a sectionally positive curvature tensor admiting such a form is positive definite, contradicting our hypothesis. This completes the proof.
\end{proof}

\noindent {\bf Decomposition of the curvature tensor}\\
Given a curvature tensor $R_{pijk}$ and a fixed metric tensor $g_{ab}$, we recall and make use of the following standard definitions:\\\\
The Ricci tensor as $R_{ij} = {R^k}_{ikj}$, and the scalar curvature as $R={R^a}_a$.\\
The scalar part as $S_{pijk}=\frac{R}{n(n-1)} (g_{pj}g_{ik}-g_{pk}g_{ij})$. \\
The traceless Ricci tensor as $S_{ab}=R_{ab}-\frac{R}{n}g_{ab}$. \\
The semi-traceless part as $E_{abcd}= \frac{1}{n-2}(g_{ac}S_{bd}+S_{ac}g_{bd}-S_{ad}g_{bc}-g_{ad}S_{bc})$.
The Schouten tensor as $P_{ab} = \frac{1}{n-2}(R_{ab}-\frac{R}{2(n-1)}g_{ab}) $.\\
The Weyl tensor, or fully-traceless part as \\$C_{abcd}= R_{abcd}-(E_{abcd}+S_{pijk})=R_{abcd}- (g_{ac}P_{bd}+P_{ac}g_{bd}-P_{ad}g_{bc}-g_{ad}P_{bc}) $.\\

\noindent These formulas are usually derived with respect to the Riemann tensor as base-component, but they can be applied to any $(4,0)$ symmetric operator on the space of $2$-forms. Let ${R_{ab}}^{cd}$ be a positive definite Riemann tensor, taken as an operator on the space of $2$-forms. We define ${\mathbb{R}^{*}_{ab}}^{cd} = ln({R_{ab}}^{cd})$. Using this, we can construct a new set of operators from ${\mathbb{R}^{*}_{ab}}^{cd}$, analogous to those defined from ${R_{ab}}^{cd}$. \mbox {For example, $W^{*}_{abcd}$ as the Weyl-component of $R^{*}_{abcd}$.}

\begin{thm}
Let ${R_{abcd}}$ be a positive-definite Riemann tensor and $n\geq 3$. Then $R_{pijk}$ admits a writing of the form $R_{pijk}=\Pi_{pj}\Pi_{ik}-\Pi_{pk}\Pi_{ij}$, for a symmetric $\Pi_{pk}$, if and only if the Weyl-component of ${\mathbb{R}^{*}_{ab}}^{cd} = ln({R_{ab}}^{cd})$ is $0$. Furthermore, the (only) resulting solutions are $\Pi_{ab}= \pm e^{P^{*}_{ab}}$, where $P^{*}_{ab}$ is the Schouten component of  ${\mathbb{R}^{*}_{ab}}^{cd}$.
\end{thm}
\begin{proof}
Assume ${R_{abcd}}$ can be written as $\Pi_{pj}\Pi_{ik}-\Pi_{pk}\Pi_{ij}$. Select a basis $(v^1,v^2,....v^n)$ that diagonalizes $\Pi_{pk}$ with eigenvalues  $(\lambda_1, \lambda_2,...,\lambda_n)$. The eigenvectors of $R_{pijk}$ can be written as $\{w^{kp}= v^k \wedge v^p \mid  k,p\in \{1,2...n \}, k<p \}$, with eigenvalues $\lambda_k\lambda_p$ for $w^{kp}$. Since the eigen-values of $R_{pijk}$ are all positive and take the form $\lambda_k\lambda_p$, then the eigen-values of any $\Pi_{ij}$ are either all positive or all negative. If they are all positive, then $\Pi_{ik}$ is uniquely determined from $R_{pikj}$ as $e^{P^{*}_{ab}}$, where $P^{*}_{ab}$ is the Schouten component of  ${R^{*}_{ab}}^{cd}$. The eigenvalues of  ${R^{*}_{ab}}^{cd}$ take the form $ln(\lambda_{k}\lambda_{p})=ln(\lambda_{k})+ln(\lambda_{p})$ for eigenvectors $v^k \wedge v^p$, and the eigenvalues of  $P^{*}_{ab}$ are $ln(\lambda_{k})$ for eigenvectors $v^k$. We recover $\Pi_{ab}$ as $e^{P^{*}_{ab}}$. If they are all negative, then multiply  $\Pi_{pj}$ by a factor of $-1$ to relate to the unique positive solution. ${\mathbb{R}^{*}_{ijkl}}$ is equal to $P^{*}_{ik}g_{jl}+g_{ik}P^{*}_{jl}-P^{*}_{il}g_{jk}+g_{il}P^{*}_{jk}$, thus having vanishing Weyl component.

Conversely, assume the Weyl component of  $R^{*}_{pikj}$ vanishes. \\Then $R^{*}_{pikj}=P^{*}_{ik}g_{jl}+g_{ik}P^{*}_{jl}-P^{*}_{il}g_{jk}-g_{il}P^{*}_{jk}$, and ${R_{pi}}^{kj}=e^{({R_{pi}^{*}}^{kj})}$ satisfies\\  $R_{pijk}=\Pi_{pj}\Pi_{ik}-\Pi_{pk}\Pi_{ij}$, for  $\Pi_{ab}= e^{P^{*}_{ab}}$. This can be checked by using a basis that diagonalizes $P^{*}_{ab}$:

Let $(v^1,v^2,....v^n)$ be such a basis, with eigen-values $(ln(\lambda_1), ln(\lambda_2),....ln(\lambda_n))$. The eigen-vectors of $R^{*}_{pikj}=P^{*}_{ik}g_{jl}+g_{ik}P^{*}_{jl}-P^{*}_{il}g_{jk}+g_{il}P^{*}_{jk}$ take the form $w^{kp}= v^k \wedge v^p$ with eigen-values $ln(\lambda_{k}\lambda_{p})=ln(\lambda_{k})+ln(\lambda_{p})$. We check that ${R_{pi}}^{kj}=e^{({R_{pi}^{*}}^{kj})}$ satisfies  $R_{pijk}=\Pi_{pj}\Pi_{ik}-\Pi_{pk}\Pi_{ij}$, for  $\Pi_{ab}= e^{P^{*}_{ab}}$, by verifying the two sides have the same eigen-vectors with the same corresponding eigen-values.

Due to the previous sections, $\pm \Pi_{ab}$ are the only solutions.
\end{proof}

The uniqueness of $\Pi_{ab}$ implies that for $n\geq 3$ and positive sectional curvature an isometric immersion, if it exists, it is rigid, \cite[page 4]{LW}.

Attempting to extend the procedure to arbitrary curvature runs into problems, as the logarithm function is not uniquely defined on negative-valued operators. If $n\geq 3$ and $R_{abcd}$ is of rank greater than one, then it is still true that $R_{abcd}= \Pi_{pj}\Pi_{ik}-\Pi_{pk}\Pi_{ij}$ has either $0$ or $2$ solutions, with $\Pi_1 = - \Pi_2$. In the case of $n=2$, there are an infinite number of solutions. They satisfy $det(\Pi)=R$. Negative curvature operators $R_{abcd}$ for $n\geq 3$ do not have (real) solutions, as they take the form $i\Pi_{pj}$, with $\Pi_{pj}$ being a solution for the positive operator $-R_{abcd}$. Although manifolds of negative curvature and dimension greater than $2$ do not admit local isometric immersions in $\mathbb{R}^{n+1}$, they may admit immersions in a pseudo-Euclidean space of signature $(n,1)$ (see end of Section 2).

\begin{thm}
For $n\geq 3$, a Riemannian $n$-manifold $M$ of positive sectional curvature admits a local isometric immersion in $\mathbb{R}^{n+1}$ if and only if it is of positive curvature operator and the following conditions are satisfied: $C^{*}_{abcd}=0$, and $e^{P^{*}}_{a[b;c]}=0$, where $C^{*}_{abcd}$ and $P^{*}_{ab}$ are the Weyl and Schouten tensors of ${\mathbb{R}^{*}_{ab}}^{cd}=ln({R_{ab}}^{cd})$.
\end{thm}
\begin{proof} Combine the fundamental theorem of hypersurfaces (Theorem 3.2) with Lemma 4.2 and Theorem 4.4:
The fundamental theorem of hypersurfaces states that for $M$ to admit an isometric immersion in $\mathbb{R}^{n+1}$ is equivalent to the existence of $\Pi$ satisfying the Gauss and Codazzi equations.  Lemma 4.2 states that manifolds of sectionally positive curvature satisfying Gauss equation have positive curvature operators. Theorem 4.4 (applied for manifolds of positive curvature)  allows us to rewrite the Gauss and Codazzi equations in the forms stated above: the existence of  $\Pi$ satisfying the Gauss equation becomes $C^{*}_{abcd}$ = 0, with $\Pi_{ab}$ uniquely determined as $\pm e^{P^{*}}_{ab}$.
\end{proof}

For the case of $n=3$, this recovers \cite[Theorem 7]{LW}.
\begin{lem}
For $n>3$ and $M$ a Riemannian manifold of positive sectional curvature, $C^{*}_{abcd} = 0$ implies \hspace*{0.0ex}  $e^{P^{*}}_{a[b;c]}=0$. That is, for a manifold of positive sectional curvature and $n>3$, the existence of $\Pi_{ab}$ satisfying the Gauss equation implies $\Pi_{ab}$ also satisfies the Codazzi equation.
\end{lem}
\begin{proof}
From Theorem 4.4 we know that  $C^{*}_{abcd} = 0$  implies the existence of a tensor $\Pi_{ab}$ such that $R_{abcd}=\Pi_{ac}\Pi_{bd}-\Pi_{ad}\Pi_{bc}$, and $\Pi_{ab}= e^{P^{*}_{ab}}$. Note that $\Pi$ is invertible.  Our aim is to prove that $\Pi_{a[b;c]} = 0$. \\The second Bianchi identity says: $R_{abcd;e} + R_{abde;c} + R_{abec;d}=0$; \\ Denote $T_{abcde}  = R_{abcd;e} + R_{abde;c} + R_{abec;d}$ and $Y_{abc}= \Pi_{a[b;c]}$.\\ Substituting $R_{abcd}=\Pi_{ac}\Pi_{bd}-\Pi_{ad}\Pi_{bc}$ in $T_{abcde}$ and grouping, we obtain:\\
$T_{abcde}=4(\Pi_{d\lbrack b}Y_{a\rbrack ce} + \Pi_{c\lbrack b}Y_{a\rbrack ed} + \Pi_{e\lbrack b}Y_{a\rbrack dc})$.\\
Let $T_{bde} = T_{abcde}(\Pi^{-1})^{ac}$ and $T_e= T_{bde}(\Pi^{-1})^{bd}$. We have:\\
$T_{bde} = T_{abcde}(\Pi^{-1})^{ac} = (n-3)Y_{bde} + 2Y_{ac\lbrack e}\Pi_{d\rbrack b}(\Pi^{-1})^{ac}$\\
$T_e= T_{bde}(\Pi^{-1})^{bd} = 2(n-2)Y_{ace} (\Pi^{-1})^{ac}$.\\
$(n-2)(n-3)Y_{bde}=(n-2)T_{bde}-T_{\lbrack e}\Pi_{d\rbrack b}$.\\
$(n-2)(n-3)\Pi_{b[d;e]}=(n-2)T_{bde}-T_{\lbrack e}\Pi_{d\rbrack b}$.\\
Since $T_{abcde}$, $T_{bde}$ and $T_e$ are $0$, we know $Y_{bde}=\Pi_{a[b;c]} = 0$ for $n>3$, thus completing our proof.
\end{proof}

\section{ Similarities to the Weyl-Schouten theorem and conformally flat manifolds}
\noindent Combining Theorem 4.5 and Lemma 4.6, we obtain the following theorem:
\begin{thm}
For $n\geq 3$, a Riemannian $n$-manifold $M$ of positive curvature operator admits a local isometric immersion in $\mathbb{R}^{n+1}$ if and only if,

when $n=3$, $e^{P^{*}}_{a[b;c]}=0$, and

when $n>3$, $C^{*}_{abcd}=0$,

\noindent where $C^{*}_{abcd}$ and $P^{*}_{ab}$ are the Weyl and Schouten tensors of ${\mathbb{R}^{*}_{ab}}^{cd}=ln({\mathbb{R}^{*}_{ab}}^{cd})$.
\end{thm}
\noindent Compare this to the Weyl-Schouten theorem:
\begin{thm}[Weyl-Schouten]
For $n\geq 3$, a Riemannian $n$-manifold is conformally flat if and only if,

when $n=3$, $P_{a[b;c]}=0$, and

when $n>3$, $C_{abcd}=0$,

\noindent where $C_{abcd}$ and $P_{ab}$ are the Weyl and Schouten tensors of $R_{abcd}$.
\end{thm}
Note that when $n=2$, every Riemannian manifold is conformally flat. Analogously, when $n=2$, every Riemannian manifold admits a local isometric immersion in $\mathbb{R}^3$ (Cartan-Janet theorem).

\section{Cross-sections and new structures resulting from immersions}
The existence of immersions of an $n$-manifold $M$ in $\mathbb{R}^{n+k}$ allows for the study of new geometric structures. We can associate to $M$ a new notion of distance, the Euclidean distance of the  immersed points in $\mathbb{R}^{n+k}$. When a global immersion exists, this distance is smooth globally, unlike the intrinsic distance which is not smooth near its cut locus. Additionally we can study a new class of $p$-submanifolds of $M$, those that appear as the intersection of $M$'s immersion with various $(n+p)$-hyperplanes of $\mathbb{R}^{n+k}$.

For our particular approach when $k=1$, the $(n-1)$-manifolds that are the cross-sections of $M$ by $n$-planes are the equipotential lines of $h$. If two points have the same $h$, the Euclidean distance between them is the distance measured along the resulting flat metric $f_{ab}=g_{ab}-h_{;a}h_{;b}$. For any two points $a$ $b$, there exists a choice of $h$ such that $h(a)=h(b)$.

We shall calculate the curvature tensor of cross-section $(n-1)$-submanifolds, or the manifolds for which $h$ is constant, for a particular $h$.
Let $N$ be such a submanifold, and $R_N$ be its curvature tensor. We shall use $a,b,c,d\in \{1,2..n\}$ as to index over $TM$, and  $i,j,k,l\in \{1,2..n-1\}$ to index over $TN$. We know that
$$ (R_{N})_{ijkl} = (R_{M})_{ijkl} +(K_{ik}K_{jl}-K_{il}K_{jk})$$
where $K$ is the second fundamental form of $N$ as a submanifold on $M$,  $R_{M}$ is the curvature tensor of $M$.

 We may compute, $K_{ij}$ as the restriction of $n_{a;b}$ to $TN$, where $n_{a}$ is a field of unit co-vectors normal to $N$. Since $h$ is constant on $N$, $h_{;a}$ is already a field of normal vectors. We shall denote $P^{a}_{i}$ to be the linear map from $TN$ to $TM$ over $N$, ie. the differential of $N$'s immersion map.

 Let $n_{a}=\frac{h_{;a}}{(g^{mn}h_{;m}h_{;n})^{\frac{1}{2}}}$ be our field of unit normal co-vectors. This gives:\\
 $K_{ij}=n_{d;c}P^{d}_{i}P^{c}_{j}  = \frac{1}{(g^{mn}h_{;m}h_{;n})^{\frac{1}{2}}}h_{;dc}P^{d}_{i}P^{c}_{j}=\frac{(1-g^{mn}h_{;m}h_{;n})^{\frac{1}{2}}}{(g^{mn}h_{;m}h_{;n})^{\frac{1}{2}}}\Pi_{cd}P^{d}_{i}P^{c}_{j}\implies  $\\ $K_{ij}=\frac{(1-g^{mn}h_{;m}h_{;n})^{\frac{1}{2}}}{(g^{mn}h_{;m}h_{;n})^{\frac{1}{2}}}\Pi_{ij}$. The second fundamental form $K_{ij}$ is the \mbox{restriction} of  $\frac{(1-g^{mn}h_{;m}h_{;n})^{\frac{1}{2}}}{(g^{mn}h_{;m}h_{;n})^{\frac{1}{2}}}\Pi_{ab}$ to $TN$. Plugging this in the curvature equation:\\  $(R_{N})_{ijkl} = (R_{M})_{ijkl} +(K_{ik}K_{jl}-K_{il}K_{jk})$, we get: $(R_{N})_{ijkl}= \frac{1}{g^{mn}h_{;m}h_{;n}}(R_{M})_{ijkl}$.

The curvature tensor of the cross-section manifold, $(R_{N})_{ijkl}$, is simply the restriction/pullback of the curvature tensor of our ambient manifold, $(R_{M})_{abcd}$, with a scaling factor of  $\frac{1}{g^{mn}h_{;m}h_{;n}}$.

\section{Conclusion}

We have calculated the obstructions to the existence of local isometric immersions in $\mathbb{R}^{n+1}$ of $n$-manifolds having positive-sectional curvature. The resulting theorem has a structural similarity to Weyl-Schouten theorem, hinting at a relationship between the two classes of manifolds: conformally flat and locally-hypersurfaces of $\mathbb{R}^{n+1}$. Potential further areas of study are finding links that make this relationship explicit or derive it as a consequence of some more general principle, reductions to obstruction-form of immersion equations when codimension is higher than $1$, and immersions in various non-Euclidean spaces or with various structures (eq: immersions in pseudo-Euclidean spaces, complex spaces, immersions of K{\"a}hler manifolds, etc.)

\vspace*{0.3cm}
\footnotesize {\noindent Dan Fodor\\
Faculty of Mathematics\\
 Alexandru Ioan Cuza University\\
 700506, Ia\c si\\
 Romania\\
 Email: \texttt{\scriptsize{dan.fodor52@yahoo.com}}}


\begin{thebibliography}{99}
\small{
\bibitem{Fodor}{Fodor, D.:} \emph{Expressing the curvature tensor and connection of a given metric in terms of those of another metric}, arXiv:1801.07122.
\bibitem{LW}{Li, YY; Weinstein, G:} \emph{A priori bounds for co-dimention one isometric embeddings}, American Journal of Mathematics, \textbf{121}(5), 1999, 945--965.
\bibitem{KN}{Kobayashi, S; Nomizu, K:} \emph{Foundations of differential geometry}, Wiley Classic Library,  vol 2, 1996.
\bibitem{Hatcher}{Hatcher, A:}  \emph{Vector Bundles and K-Theory}, \\ https://www.math.cornell.edu/\textasciitilde{}hatcher/VBKT/VB.pdf
\bibitem{HH}{Han, Q.; Hong, J-X.:} \emph{Isometric Embedding of Riemannian Manifolds in Euclidean Spaces}, American Mathematica Society, 2006.}

\end{thebibliography}
\end{document}